\documentclass[reqno]{amsart} 
\usepackage{amssymb}
\usepackage{amsmath}
\usepackage{graphicx,color}
\usepackage{t1enc}
\usepackage[utf8]{inputenc}

\usepackage{amsthm}

\usepackage{enumitem}

\usepackage{latexsym}
\usepackage{subfigure}
\usepackage{tikz}
\usepackage{float}
\usepackage{graphicx,color}
\newtheorem{theorem}{Theorem}[section]
\newtheorem{definition}{Definition}[section]
\newtheorem{lemma}{Lemma}[section]
\newtheorem{proposition}{Proposition}[section]
\newtheorem{remark}{Remark}[section]

\newtheorem{corollary}{Corollary}[section]

\newcommand\restr[2]{{
		\left.\kern-\nulldelimiterspace 
		#1 
		\vphantom{\big|} 
		\right|_{#2} 
}}

\begin{document}

\title[ Existence and non-existence of positive solutions]{A new result on the existence and non-existence of positive solutions for two-parametric systems of quasilinear elliptic equations}

\author{\bf\large R.L. Alves }
 

\maketitle
\begin{center}
	{\bf\small Abstract}
	
	\vspace{3mm} \hspace{.05in}\parbox{4.5in} {{\small This paper is devoted to the existence and non-existence of positive solutions for a $(p, q)$-Laplacian system with indefinite nonlinearity depending on two parameters $(\lambda,\mu)$. By using the sub-supersolution method together with adaptations of ideas found in \cite{AAS,CRMZ}, we extend some previous result.	 } }
\end{center}

\noindent


%
%
%
\section{\bf Introduction}
\def\theequation{1.\arabic{equation}}\makeatother
\setcounter{equation}{0}
In this paper we consider the system of equations
\def\theequation{1.\arabic{equation}}\makeatother
\setcounter{equation}{0}
\begin{equation*}\label{pq}\tag{$P_{\lambda,\mu}$}
\left\{
\begin{aligned}
&-\Delta_{p} u = \lambda \vert u\vert^{p-2}u+\alpha f(x) \vert u\vert^{\alpha-2}\vert v\vert ^{\beta}u~in ~ \Omega ,\\
&-\Delta_{q} v =  \mu \vert  v\vert^{q-2}v+\beta f(x)\vert u\vert ^{\alpha}\vert v\vert^{\beta-2}v~in~ \Omega ,\\
&(u,v)\in W^{1,p}_{0}(\Omega)\times  W^{1,q}_{0}(\Omega),\\ 
&u,v>0~\mbox{in}~\Omega, \nonumber
\end{aligned}
\right.
\end{equation*}
where $\Omega \subset \mathbb{R}^{N}$ ($N\geq 3$) is a smooth bounded domain;  the function $f\in L^{\infty}(\Omega)$ and has indefinite sign, that is $f^{+}=\max \left\{f,0\right\}$ and $f^{-}=\min \left\{f,0\right\}$ are not identically zero in $\Omega$; $\lambda,\mu\in \mathbb{R}$; $\Delta_{r}$ ($r\in \left\{p,q\right\}$) is the $r$-Laplacian operator; $1<p,q<\infty$,
\begin{equation}\label{pot}
\alpha\geq p,\beta\geq q~\mbox{and}~\frac{\alpha}{p}+\frac{\beta}{q}>1,~~~~\frac{\alpha}{p^{\ast}}+\frac{\beta}{q^{\ast}} <1,
\end{equation}
where $p^{\ast}$ and $q^{\ast}$ are the critical Sobolev exponents of $W^{1,p}_{0}(\Omega)$ and $W^{1,q}_{0}(\Omega)$.


We will use the symbols $(\hat{\lambda}_{1},\phi_{1})$ and $(\hat{\mu}_{1},\psi_{1})$ for the first eigenpairs of the operators $-\Delta_{p}$ and $-\Delta_{q}$ in $\Omega$ with zero boundary conditions, respectively. 

Consider the product space $E=W_{0}^{1,p}(\Omega)\times W_{0}^{1,q}(\Omega)$ equipped with the norm
$$\Vert (u,v)\Vert= \left(\displaystyle \int_{\Omega}\vert \nabla u\vert^{p}\right)^{\frac{1}{p}}+ \left(\displaystyle \int_{\Omega}\vert \nabla v\vert^{q}\right)^{\frac{1}{q}}:=\Vert u\Vert_{p}+\Vert v\Vert_{q},~(u,v)\in E.$$

We say that $(u,v)\in E$ is a positive solution of \eqref{pq} if $u>0$, $v>0$ in $\Omega$ and 
\begin{align*}
\displaystyle \int_{\Omega}\vert \nabla u\vert^{p-2}\nabla u \nabla \phi + \displaystyle \int_{\Omega}\vert \nabla v\vert^{q-2} \nabla v \nabla\psi- \displaystyle \lambda \displaystyle \int_{\Omega}\vert u\vert^{p-2}u\phi-\alpha \displaystyle \int_{\Omega}f(x) \vert u\vert^{\alpha-2}\vert v\vert ^{\beta}u\phi\\
- \mu \displaystyle \int_{\Omega} \vert  v\vert^{q-2}v\psi-\beta \displaystyle \int_{\Omega}f(x)\vert u\vert ^{\alpha}\vert v\vert^{\beta-2}v\psi=0, \forall (\phi,\psi)\in E.
\end{align*}

Thus, the corresponding energy functional of problem \eqref{pq} is defined by
$$I_{\sigma}(u,v)=\frac{1}{p}\left(\Vert u \Vert^{p}_{p}-\lambda\vert u\vert_{p}^{p}\right)+\frac{1}{q}\left(\Vert  v \Vert^{q}_{q}-\mu\vert v\vert_{q}^{q}\right)-F(u,v), (u,v)\in E,$$
where $\sigma=(\lambda,\mu)$, $F(u,v)=\displaystyle \int_{\Omega}f\vert u\vert^{\alpha}\vert v\vert^{\beta}$ and $\vert . \vert_{p},\vert . \vert_{q}$ are the standard $L^{p}(\Omega)$ and $L^{q}(\Omega)$ norm. A pair of functions $(u,v)\in E$ is said to be a weak solution of \eqref{pq} if $(u,v)$ is a critical point of $I_{\sigma}$. Let us observe that if $\sigma=(\hat{\lambda}_{1},\mu),\mu \in \mathbb{R}$ (respectively  $\sigma=(\lambda,\hat{\mu}_{1}),\lambda \in \mathbb{R}$), then $(\epsilon \phi_{1},0 ),\epsilon \in \mathbb{R}$ (respectively $(0,\epsilon \psi_{1} ),\epsilon \in \mathbb{R}$) is a weak solution of $(P_{\hat{\lambda}_{1},\mu})$ (respectively ($P_{\lambda,\hat{\mu}_{1}}$)).

Under suitable assumptions on the function $f$ (see assumptions $(f)_{1}-(f)_{2}$ bellow), and by using the fibering Method of Pohozaev, problem \eqref{pq} has been studied by Bozhkov-Mitidieri \cite{BM}, Bobkov-Il'yasov \cite{BI1,BI2}
and Silva-Macedo \cite{SM1}. From these works we know that:
\begin{itemize}
	\item if $(\lambda,\mu)\in (-\infty,\hat{\lambda}_{1}[\times (-\infty,\hat{\mu}_{1}[$, then problem \eqref{pq} admits at least one positive solution (see \cite{BI1,BM});
	\item there exist curves $\Gamma_{\ast} \subset \Pi:=\left\{(\lambda,\mu): \hat{\lambda}_{1}<\lambda~\mbox{and}~\hat{\mu}_{1}<\mu \right\}$, $\Gamma^{\ast}\subset \left\{(\lambda,\mu):0<\lambda~\mbox{and}~0<\mu\right\}$ which determine regions $\Pi_{1},\Pi_{2}\subset \Pi$ satisfying $\Gamma_{\ast}\cap \Pi_{1}=\emptyset, \Gamma^{\ast}\cap \Pi_{2}=\emptyset$ such that problem \eqref{pq} has no
	positive solution and has at least one according to $(\lambda,\mu)$ belongs to $\Pi_{2}$ and
	$\Pi_{1}$, respectively (see \cite{BI1,BI2,SM1}).
	
	These results are depicted in the following figure. We observe that they showed the existence of a positive solution only in the blue region.
\end{itemize}

\begin{center}
	\begin{tikzpicture}
	\draw[thick,->] (-2,0) -- (4.5,0);
	\draw[thick,->] (0,-2) -- (0,4.5);
	\draw[dashed] (-2,1) -- (4.5,1);
	\draw[dashed] (1,-2) -- (1,4.5);
	\filldraw[fill=blue, draw=black, thick,dashed] (1.005,1.005)--(2,1) arc (1:75:1.35cm);
	\draw[thick,dashed] (3,0) arc (0.2:75:4.08cm);
	\fill[blue] (-1.999,-1.9999) rectangle (0.99,0.98);
	\draw (4.5,0) node[below]{$\lambda$};
	\draw (0,4.4) node[left]{$\mu$};
	\draw (0.9,0) node[below right]{$\hat{\lambda}_{1}$};
	\draw (0,0.9) node[above left]{$\hat{\mu}_{1}$};
	\draw (2,1.5) node[above ]{$\Gamma_{\ast}$};
	\draw (2.5,2.2) node[above ]{$\Gamma^{\ast}$};
	\draw (3,3) node[above ]{$\Pi_{2}$};
	\draw (0,0) node[above ]{$\Pi_{1}$};
	\draw (1.5,1.3) node[above ]{$\Pi_{1}$};
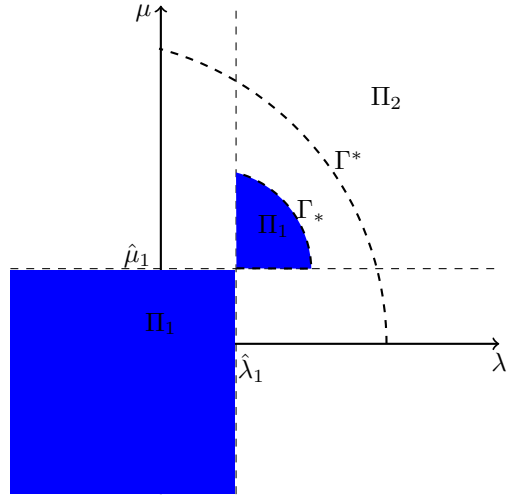
	\draw (1,-2) node[below]{\textrm{Figure 1. $\Pi_{1}=$ blue region; $\Pi_{2}=\left\{(\lambda,\mu):(0,0)<\Gamma^{\ast}(t)<(\lambda,\mu)\right\}$}};
	
	\end{tikzpicture}
\end{center}


The  main  novelty  in  the  present  paper  is  the  investigation  of  the  existence  and nonexistence of positive solution outside the sets $\Pi_{1}$ and $\Pi_{2}$ (see Figure 1). Special attention is paid in finding of the curve $\Gamma \subset \mathbb{R}^{2}, \Gamma(t)=(\lambda(t),\mu(t)),t\in \mathbb{R}$ called to be the extremal curve such that problem \eqref{pq} admits a positive solution if $\lambda<\lambda(t),\mu<\mu(t),t\in \mathbb{R}$ and $(\lambda,\mu)\notin \Lambda_{1}\cup \Lambda_{2}$, while if $\lambda>\lambda(t)$ and $\mu>\mu(t),t\in \mathbb{R}$ then \eqref{pq} admits no positive solution (see \eqref{Gam} for the definition of the sets $\Lambda_{1}$ and $\Lambda_{2}$). 

Our proof is based on the super and sub-solution method and some adaptations of the ideas found in the works of the author et al. \cite{AAS,CRMZ} and Cheng-Zhang \cite{CZ}. As in \cite{BM,BI1,BI2,SM1}, throughout this paper we will
assume the following hypothesis
\begin{itemize}
	\item[$(f)_{1}$] $F(\phi_{1},\psi_{1})=\displaystyle \int_{\Omega}f\vert \phi_{1}\vert^{\alpha}\vert \psi_{1}\vert^{\beta}<0$;
	\item[$(f)_{2}$] Let $\Omega^{0}=\left\{x\in \Omega:f(x)=0\right\}$ and $\Omega^{+}=\left\{x\in \Omega:f(x)>0\right\}$. The measure of  $\Omega^{0}$ is not zero and the interior of  $\Omega^{0}\cup \Omega^{+}$ is a regular domain. Moreover $\tilde{\Omega}=int(\Omega^{0}\cup \Omega^{+})$ contains a connected component which intersects both $\Omega^{0}$ and $\Omega^{+}$.
\end{itemize}

Before stating our main results, we introduce the following two sets
\begin{equation}\label{Gam}
\Lambda_{1}=\left\{(\lambda,\hat{\mu}_{1})\in \mathbb{R}^{2}: \lambda\leq \hat{\lambda}_{1} \right\}~\mbox{and}~ \Lambda_{2}=\left\{(\hat{\lambda}_{1},\mu)\in \mathbb{R}^{2}: \mu\leq \hat{\mu}_{1} \right\}.
\end{equation}



Our first result is the following.

\begin{theorem}\label{T1}
	Assume \eqref{pot} and $(f)_{1}-(f)_{2}$ hold.	There exist $\lambda_{\ast},\mu_{\ast}\in \mathbb{R}$ and a continuous simple arc $\Gamma$, satisfying $\displaystyle \lim_{t\to +\infty}\Gamma(t)=(\lambda_{\ast},-\infty)$ and $\displaystyle \lim_{t\to -\infty}\Gamma(t)=(-\infty,\mu_{\ast})$ that separates $\mathbb{R}^{2}$
	into two disjoint open subsets $\Theta_{1}$ and $\Theta_{2}$, with $\Lambda_{1}\cup \Lambda_{2}\subset \Theta_{1}$ such that the system \eqref{pq} has no
	positive solution and has at least one according to $(\lambda,\mu)$ belongs to $\Theta_{2}$ and
	$\Theta_{1}\backslash \left(\Lambda_{1}\cup \Lambda_{2}\right)$, respectively. Moreover, $\partial \Theta_{1}\backslash \left(\Lambda_{1}\cup \Lambda_{2}\right)=\Gamma$.
\end{theorem}

Theorem \ref{T1} partially extend the main results in \cite{BI1,BI2,SM1}, because it show the existence of positive solution of \eqref{pq} for every $(\lambda,\mu)\in \Theta_{1}\backslash \left(\Lambda_{1}\cup \Lambda_{2}\right)$, where $\Upsilon \cap \left(\mathbb{R}^{2}\backslash \left(\Theta_{1}\cup \Gamma\right)\right)=\emptyset$ (see \eqref{up} for the definition of $\Upsilon$). In particular, it shows the existence of positive solution for $\hat{\lambda}_{1}<\lambda,\mu<\hat{\mu}_{1}$ and $\lambda<\hat{\lambda}_{1},\hat{\mu}_{1}<\mu$, too. Furthermore, Theorem \ref{T1} gives us an almost complete description of the set of parameters $(\lambda,\mu)$ such that \eqref{pq} admits a positive solution (see Figure. 2). In particular, we would like to point out that the curves $\Gamma$ and $\Gamma^{\ast}$ are distinct (see Figures 1 and 2).

\begin{center}
	\begin{tikzpicture}
	\draw[thick,->] (-2,0) -- (4.9,0);
	\draw[thick,->] (0,-2) -- (0,4.9);
	\draw[dashed] (-2,0.98) -- (1,0.98);
	\draw[dashed] (0.975,-2) -- (0.975,1);
	\draw[dashed] (4.1,-2) -- (4.1,0);
	\draw[dashed] (-2.1,4.04) -- (0,4.04);
	\draw (4,0) node[below right]{$\lambda_{\ast}$};
	\draw (0,3.9) node[above right]{$\mu_{\ast}$};

	\draw[fill=blue, draw=black,thick,dashed] (4.1,-2) .. controls (3,3) and (0,4)..(-2.1,4);
	\fill[blue] (-2.1,4) rectangle (-1.9,1.005);
	
	\fill[blue] (-1.9,3.8) rectangle (-1,1.005);

	\fill[blue] (-1,3.5) rectangle (0,1.005);
	
	\fill[blue] (-1,3) rectangle (1,1.005);
	
	\fill[blue] (1,0.95) rectangle (3,-2);
	
	\fill[blue] (4.088,-2) rectangle (1,-1.95);
	
	\fill[blue] (4.05,-1.95) rectangle (1,-1.9);
	
	\fill[blue] (3.987,-1.5) rectangle (1,-1.9);
	
	\fill[blue] (3.8,-0.9) rectangle (1,-1.9);
	
	\fill[blue] (-2.1,-2) rectangle (0.95,0.95);
	
	\fill[blue] (1,0.9) rectangle (2,2);
	
	\draw (4.9,0) node[below]{$\lambda$};
	\draw (0,4.8) node[left]{$\mu$};
	\draw (0.8,-1.9) node[below right]{$\Lambda_{2}$};
	\draw (-2,0.8) node[above left]{$\Lambda_{1}$};
	\draw (2.1,2.2) node[above ]{$\Gamma$};
	
	\draw (3,3) node[below ]{$\Theta_{2}$};
	\draw (-2.3,-1.5) node[below ]{$\Theta_{1}$};
	\draw (1,-2.4)node[below]{\textrm{Figure 2. $\Theta_{1}=$ blue region; $\Theta_{2}=\left\{(\lambda,\mu):\Gamma(t)<(\lambda,\mu)\right\}$}};
	\end{tikzpicture}
\end{center}

The next result deals with the existence and multiplicity of positive solutions of \eqref{pq} when $p=q=2$ and $\lambda=\mu$. 
\begin{theorem}\label{T2}
	Let $p=q=2$, $f\in C(\overline{\Omega})$, and assume that \eqref{pot} and $(f)_{1}-(f)_{2}$ hold. Then there exists $\tau>0$ such that
	\begin{itemize}
		\item[$a)$] for every $\lambda\in (\hat{\lambda}_{1},\tau)$, $(P_{\lambda,\lambda})$ admits at least two positive solutions.
		\item[$b)$] for $\lambda=\hat{\lambda}_{1}$ and $\lambda=\tau$, problem $(P_{\lambda,\lambda})$ admits at least one positive solution.
	\end{itemize} 
\end{theorem}

This paper is organized as follows. In Section 2, we give some definitions and 
prove a sub-supersolution theorem. In Section 3 we study some properties of the set of parameters $(\lambda,\mu)$ such that \eqref{pq} admits  at least one positive solution, and then give the proof of Theorem \ref{T1}. In Section 4, we prove Theorem \ref{T2}.

{\it Notations.} Throughout this paper, we make use of the following notations.
\begin{itemize}
	\item The spaces $\mathbb{R}^{N}$ are equipped with the Euclidean norm $\sqrt{x^{2}_{1}+\cdot \cdot \cdot + x^{2}_{N}} $,
	\item $B_{r}(x)$ denotes the ball centered at $x \in \mathbb{R}^{N}$ with radius $r>0$,
	\item the notation $(a,b)>(c,d)$ means $a>c$ and $b>d$. Similarly, $(a,b)\geq(c,d)$ means $a\geq c$ and $b\geq d$ for all $(a,b),(c,d)\in \mathbb{R}^{2}$, 
	\item  $\displaystyle \lim_{|x|\to \infty}(u(x),v(x))=(\displaystyle \lim_{|x|\to \infty}u(x),\displaystyle \lim_{|x|\to \infty}v(x))$ for  functions $u,v:\mathbb{R}^{N}\longrightarrow \mathbb{R}$, 
	\item If $A$ is a measurable set in $\mathbb{R}^{N}$, we denote by $\mathcal{L}(A)$  the Lebesgue measure of $A$.
\end{itemize}

\section{Sub-supersolution theorem for $(\lambda,\mu)\notin ]-\infty,\hat{\lambda}_{1}]\times ]-\infty,\hat{\mu}_{1}]$}
In this section we will give some definitions and prove a sub-supersolution theorem
that will be essential to prove Theorem \ref{T1}.

Let us consider the space $C^{1}_{0}(\overline{\Omega})=\left\{u\in C^{1}(\overline{\Omega}): u=0~\mbox{on}~ \partial \Omega \right\}$ equipped with the norm $\Vert u \Vert_{C^{1}}=\displaystyle\max_{x\in \Omega}\vert u(x)\vert+\displaystyle\max_{x\in \Omega}\vert \nabla u(x)\vert$. If on $C^{1}_{0}(\overline{\Omega})$ we consider the pointwise partial ordering (i.e., $u\leq v$ if and only if $u(x)\leq v(x)$ for all $x\in \overline{\Omega}$), which is induced by the positive cone
$$C^{1}_{0}(\overline{\Omega})_{+}=\left\{u\in C^{1}_{0}(\overline{\Omega}): u\geq 0~\mbox{for all}~ x\in \Omega \right\},$$
then this cone has a nonempty interior given by
$$int ~C^{1}_{0}(\overline{\Omega})_{+}=\left\{u\in C^{1}_{0}(\overline{\Omega}): u> 0~\mbox{for all}~ x\in \Omega~\mbox{and}~\frac{\partial u}{\partial n}(x) <0~\mbox{for all}~ x\in \partial\Omega\right\},$$
where $n(x)$ is the outward unit normal vector to $\partial \Omega$
at the point $x\in \partial \Omega$. It is known that $\phi_{1},\psi_{1}\in int ~C^{1}_{0}(\overline{\Omega})_{+}$.

The next proposition has been proved in  Marano-Papageorgiou ( see \cite{MP}, Proposition 1).
\begin{proposition}\label{pro1}
	If $u\in int~C^{1}_{0}(\overline{\Omega})_{+}$ then to every $v\in C^{1}_{0}(\overline{\Omega})_{+}$ there corresponds $\epsilon_{v}>0$ such that $u-\epsilon_{v}v\in C^{1}_{0}(\overline{\Omega})_{+}$.
\end{proposition}

The following proposition concerning the regularity of the positive solutions will be useful in
the sequel. For proof  we refer to \cite{BI2}, Appendix 2.
\begin{proposition}\label{pro2}
	Assume \eqref{pot} holds and let $(u,v)$ be a positive solution of \eqref{pq}. Then $u,v\in C^{1}_{0}(\overline{\Omega})$. Moreover, $u$ and $v$ satisfy a boundary point maximum principle on $\partial \Omega$.
\end{proposition}

As a consequence of Proposition \ref{pro1} and Proposition \ref{pro2}, for any positive solution $(u,v)$ of \eqref{pq} one has $u\geq \epsilon \phi_{1}$, $v\geq \epsilon \psi_{1}$ in $\Omega$ for $\epsilon>0$ small enough.

Now we define the notion of sub-supersolution. 
\begin{definition}\label{def1} A pair $(\overline{u},\overline{v})\in E$  is said to be a supersolution  of \eqref{pq} if $\overline{u},\overline{v}\in int ~C^{1}_{0}(\overline{\Omega})_{+}$ and 
	\begin{align*}
	\displaystyle \int_{\Omega}\vert \nabla \overline{u}\vert^{p-2}\nabla \overline{u} \nabla \phi - \displaystyle \lambda \displaystyle \int_{\Omega}\vert \overline{u}\vert^{p-2}\overline{u}\phi-\alpha \displaystyle \int_{\Omega}f(x) \vert \overline{u}\vert^{\alpha-2}\vert \overline{v}\vert ^{\beta}\overline{u}\phi \geq 0,
	\end{align*}
	
	\begin{align*}
	\displaystyle \int_{\Omega}\vert \nabla \overline{v}\vert^{q-2} \nabla \overline{v} \nabla\psi- \mu \displaystyle \int_{\Omega} \vert  \overline{v}\vert^{q-2}\overline{v}\psi-\beta \displaystyle \int_{\Omega}f(x)\vert \overline{u}\vert ^{\alpha}\vert \overline{v}\vert^{\beta-2}\overline{v}\psi\geq 0
	\end{align*}
	for any $(\phi,\psi)\in E$ with $\phi,\psi \geq 0$ in $\Omega$.
\end{definition}

\begin{definition}\label{def2} A pair $(\underline{u},\underline{v})\in E$  is said to be a subsolution  of \eqref{pq} if 
	\begin{align*}
	\displaystyle \int_{\Omega}\vert \nabla \underline{u}\vert^{p-2}\nabla \underline{u} \nabla \phi - \displaystyle \lambda \displaystyle \int_{\Omega}\vert \underline{u}\vert^{p-2}\underline{u}\phi-\alpha \displaystyle \int_{\Omega}f(x) \vert \underline{u}\vert^{\alpha-2}\vert \underline{v}\vert ^{\beta}\underline{u}\phi \leq 0,
	\end{align*}
	
	\begin{align*}
	\displaystyle \int_{\Omega}\vert \nabla \underline{v}\vert^{q-2} \nabla \underline{v} \nabla\psi- \mu \displaystyle \int_{\Omega} \vert  \underline{v}\vert^{q-2}\underline{v}\psi-\beta \displaystyle \int_{\Omega}f(x)\vert \underline{u}\vert ^{\alpha}\vert \underline{v}\vert^{\beta-2}\underline{v}\psi\leq 0
	\end{align*}
	for any $(\phi,\psi)\in E$ with $\phi,\psi \geq 0$ in $\Omega$.
\end{definition}
\begin{remark}
	The functions $\underline{u}$ and $\underline{v}$ may not be in  $int ~C^{1}_{0}(\overline{\Omega})_{+}$.
\end{remark}

We introduce the notation $$X_{\hat{\lambda}_{1},\hat{\mu}_{1}}=\left\{(\lambda,\mu)\in \mathbb{R}^{2}: \lambda <\hat{\lambda}_{1}~\mbox{and}~ \mu <\hat{\mu}_{1} \right\}.$$

We are now ready to prove the main result of this section.
\begin{theorem}\label{sub-sup} Let $(\overline{\lambda},\overline{\mu})\notin \Lambda_{1}\cup \Lambda_{2}\cup X_{\hat{\lambda}_{1},\hat{\mu}_{1}}$. Assume \eqref{pot} holds and suppose
	that $(P_{\overline{\lambda},\overline{\mu}})$ has a supersolution. Then problem \eqref{pq} has a positive solution $(u,v)$ for each $\sigma=(\lambda,\mu)\in (-\infty,\overline{\lambda}]\times (-\infty,\overline{\mu}]
	\setminus \left(\Lambda_{1}\cup \Lambda_{2}\cup X_{\hat{\lambda}_{1},\hat{\mu}_{1}}\right)$. Moreover, one has $I_{\sigma}(u,v)<0$.
\end{theorem}

\begin{proof} We first note that $(\underline{u},\underline{v})=(0,0)$ is a subsolution of \eqref{pq} for any $(\lambda,\mu)\in \mathbb{R}^{2}$. Let $(\overline{u}, \overline{v})$ be a supersolution of $(P_{\overline{\lambda},\overline{\mu}})$ and  $(\lambda,\mu)\in (-\infty,\overline{\lambda}]\times (-\infty,\overline{\mu}]
	\setminus \left(\Lambda_{1}\cup \Lambda_{2}\cup X_{\hat{\lambda}_{1},\hat{\mu}_{1}}\right)$. Clearly, $(\overline{u}, \overline{v})$ is a supersolution of \eqref{pq} with $\overline{u}>0$ and $\overline{v}>0$ in $\Omega$. The solution of \eqref{pq} will be obtained by minimizing the functional $I_{\sigma}$ $(\sigma=(\lambda,\mu))$ over the set
	$$M=\left\{(u,v)\in E:0\leq u\leq \overline{u},~ 0\leq v\leq\overline{v} \right\}.$$
	
	We first observe that $M$ is convex and closed with respect to the $E$ topology, hence weakly closed. Furthermore, for all $(u,v)\in M$ we have
	$$I_{\sigma}(u,v)\geq \frac{1}{p}\left(\Vert \nabla u \Vert^{p}-\lambda\vert \overline{u}\vert_{p}^{p}\right)+\frac{1}{q}\left(\Vert \nabla v \Vert^{q}-\lambda\vert \overline{v}\vert_{q}^{q}\right)-\displaystyle \int_{\Omega}\vert f(x)\vert \vert \overline{u}\vert^{\alpha}\vert \overline{v}\vert^{\beta},$$	
	which implies that 	$I_{\sigma}$ is coercive on $M$. 
	
	It is easy to check that $I_{\sigma}$ is weakly lower semicontinuous on $M$. Thus, $I_{\sigma}$ verifies the hypotheses of Theorem 1.2 in \cite{St}. According to this one, there exists a relative minimizer $(u,v)\in M$ of $I_{\sigma}$. We show in what follows that  $(u,v)$ is a positive solution of \eqref{pq}.
	
	In the sequel, to simplify the notation,
	we set $U=(u,v)$,
	$$H_{1}(x,s,t)=\lambda \vert s\vert^{p-2}s+\alpha f(x)\vert s\vert^{\alpha-2}\vert t\vert^{\beta}s~ \mbox{and}~H_{2}(x,s,t)=\mu \vert t\vert^{p-2}t+\beta f(x)\vert s\vert^{\alpha}\vert t\vert^{\beta-2}t. $$
	
	Let $ (\varphi,\psi)=\Psi \in E$, $\epsilon>0$, and consider
	$$w^{\epsilon}:=\left(u+\epsilon \varphi-\overline{u}\right)^{+}=\max \left\{0,u+\epsilon \varphi-\overline{u}\right\},$$
	$$w_{\epsilon}:=\left(u+\epsilon \varphi\right)^{-}=\max \left\{0,-(u+\epsilon \varphi)\right\},$$
	$$ z^{\epsilon}:=\left(v+\epsilon \psi-\overline{v}\right)^{+}=\max \left\{0,v+\epsilon \psi-\overline{v}\right\},$$
	and
	$$z_{\epsilon}:=\left(v+\epsilon \psi\right)^{-}=\max \left\{0,-(v+\epsilon \psi)\right\}.$$
	Set $\eta_{\epsilon}:=u+\epsilon \varphi-w^{\epsilon}+w_{\epsilon}~\mbox{and}~\nu_{\epsilon}:=v+\epsilon \psi-z^{\epsilon}+z_{\epsilon}.$ Then $U_{\epsilon}:=(\eta_{\epsilon},\nu_{\epsilon})=U+\epsilon \Psi-(w^{\epsilon},z^{\epsilon})+(w_{\epsilon},z_{\epsilon})\in M,$ and by the
	convexity of $M$
	we get $U+t\left(U_{\epsilon}-U\right)\in M$, for all $0<t<1$. Since $U$ minimizes $I_{\sigma}$ in $M$, this yields
	\begin{align*} 
	&0\leq \langle I_{\sigma}^{\prime}(U),\left(U_{\epsilon}-U \right) \rangle
	=\epsilon \langle I_{\sigma}^{\prime}(U),\left(\varphi,\psi \right) \rangle -\langle I_{\sigma}^{\prime}(U),\left(w^{\epsilon},z^{\epsilon}\right) \rangle +\langle I_{\sigma}^{\prime}(U),\left(w_{\epsilon},z_{\epsilon}\right) \rangle,
	\end{align*}
	so that
	\begin{equation}\label{3121}
	\langle I_{\sigma}^{\prime}(U),\left(\varphi,\psi \right) \rangle\geq \frac{1}{\epsilon}\left[\langle I_{\sigma}^{\prime}(U),\left(w^{\epsilon},z^{\epsilon}\right) \rangle -\langle I_{\sigma}^{\prime}(U),\left(w_{\epsilon},z_{\epsilon}\right) \rangle\right].
	\end{equation}
	
	Now, since $\overline{U}=(\overline{u},\overline{v})$ is a supersolution to \eqref{pq}, we have
	\begin{align*}
	\langle I_{\sigma}^{\prime}(U),\left(w^{\epsilon},z^{\epsilon}\right) \rangle =&\langle I_{\sigma}^{\prime}(\overline{U}),\left(w^{\epsilon},z^{\epsilon}\right)\rangle+\langle I_{\sigma}^{\prime}(U)-I_{\sigma}^{\prime}(\overline{U}),\left(w^{\epsilon},z^{\epsilon}\right) \rangle\\
	&\geq  \langle I_{\sigma}^{\prime}(U)-I_{\sigma}^{\prime}(\overline{U}),\left(w^{\epsilon},z^{\epsilon}\right) \rangle \\
	&= \displaystyle \int_{\Omega_{\epsilon}}\left(\vert \nabla u\vert^{p-2}\nabla u-\vert \nabla \overline{u}\vert^{p-2}\nabla \overline{u} \right) \nabla \left(u+\epsilon \varphi-\overline{u}\right)\\
	&-\displaystyle \int_{\Omega_{\epsilon}}\left[H_{1}(x,u,v)-H_{1}(x,\overline{u},\overline{v})\right]\left(u+\epsilon \varphi-\overline{u}\right)\\
	&+  \displaystyle \int_{\Omega^{\epsilon}} \left(\vert \nabla v\vert^{q-2}\nabla v-\vert \nabla \overline{v}\vert^{q-2}\nabla \overline{v} \right)\nabla \left(v+\epsilon \psi-\overline{v}\right)\\
	&-\displaystyle \int_{\Omega^{\epsilon}}\left[H_{2}(x,u,v)-H_{2}(x,\overline{u},\overline{v})\right] \left(v+\epsilon \psi-\overline{v}\right)  \\
	&\geq  \epsilon \displaystyle \int_{\Omega_{\epsilon}}\left(\vert \nabla u\vert^{p-2}\nabla u-\vert \nabla \overline{u}\vert^{p-2}\nabla \overline{u} \right) \nabla \varphi \\
	&-\epsilon \displaystyle \int_{\Omega_{\epsilon}}\vert H_{1}(x,u,v)-H_{1}(x,\overline{u},\overline{v})\vert \vert \varphi \vert \\
	& +\epsilon \displaystyle \int_{\Omega^{\epsilon}}\left(\vert \nabla v\vert^{q-2}\nabla v-\vert \nabla \overline{v}\vert^{q-2}\nabla \overline{v} \right) \nabla \psi \\
	&-\epsilon \displaystyle \int_{\Omega^{\epsilon}}\vert H_{2}(x,u,v)-H_{2}(x,\overline{u},\overline{v})\vert \vert \psi \vert ,
	\end{align*}
	where we have used the monotonicity of $-\Delta_{p}, -\Delta_{q}$ and  
	$$\Omega_{\epsilon}=\left\{x\in \mathbb{R}^{N}:u+\epsilon \varphi\geq \overline{u}> u \right\}~\mbox{and}~\Omega^{\epsilon}=\left\{x\in \mathbb{R}^{N}:v+\epsilon \psi\geq \overline{v}> v \right\}.$$
	
	Note that $\mathcal{L}(\Omega_{\epsilon})\to 0$ and $\mathcal{L}(\Omega^{\epsilon})\to 0$ as $\epsilon \to 0$. Hence, by absolute continuity of the Lebesgue integral, one has
	\begin{equation}\label{3122}
	\frac{\langle I_{\sigma}^{\prime}(U),\left(w^{\epsilon},z^{\epsilon}\right) \rangle}{\epsilon}\geq o(\epsilon), ~\mbox{where}~o(\epsilon)\to 0~\mbox{as}~\epsilon \to 0.
	\end{equation}
	
	By using that $(0,0)$ is a subsolution and following similar arguments as done in the proof of \eqref{3122}, we get
	\begin{equation}\label{3c123}
	\frac{\langle I_{\sigma}^{\prime}(U),\left(w_{\epsilon},z_{\epsilon}\right) \rangle}{\epsilon}\leq  o(\epsilon), ~\mbox{where}~o(\epsilon)\to 0~\mbox{as}~\epsilon \to \infty,
	\end{equation}
	Putting together \eqref{3121}, \eqref{3122} and \eqref{3c123} we deduce $\langle I_{\sigma}^{\prime}(U),\Psi \rangle\geq 0$,
	for all $\Psi=\left(\varphi,\psi \right)\in  E$. Reversing the sign of $\Psi$ we find $\langle I_{\sigma}^{\prime}(U),\Psi \rangle =0$, for all $\Psi \in E$, that is, $U=(u,v)$ is a weak solution of \eqref{pq}.
	
	We claim that $(u,v)$ is a positive solution of \eqref{pq}. Indeed, one has $u,v\geq 0$ in $\Omega$ and by the assumption either $\lambda>\hat{\lambda}_{1}$ or $\mu>\hat{\mu}_{1}$. Without loss of generality we can assume $\lambda>\hat{\lambda}_{1}$ (the proof in the second case is similar). From Proposition \ref{pro1} and Proposition \ref{pro2} there exists $\epsilon>0$ such that  $(\epsilon \varphi_{1},0)\in M$, and consequently
	$$I_{\sigma}(u,v)\leq I_{\sigma}(\epsilon \varphi_{1},0)=\frac{1}{p}\left(\Vert \epsilon\varphi_{1}\Vert^{p} - \lambda \vert\epsilon\varphi_{1}\vert^{p}_{p}\right)<\frac{1}{p}\left(\Vert \epsilon\varphi_{1}\Vert^{p} - \hat{\lambda}_{1}\vert\epsilon\varphi_{1}\vert^{p}_{p}\right)=0,$$
	that is, $(u,v)\neq (0,0)$.
	
	We show in what follows that $v\neq 0$.  Assume the contrary, namely $v=0$. Since $(u,v)\neq (0,0)$, we obtain $u\geq 0$ in $\Omega$, $u\neq 0$ and
	$$\displaystyle \int_{\Omega_{\epsilon}}\vert \nabla u\vert^{p-2}\nabla u \nabla \varphi=\lambda\vert u\vert^{p-2}u\varphi,$$	
	for every $\varphi \in W^{1,p}_{0}(\Omega)$. By standard regularity  arguments and the strong maximum principle, we deduce that $u\in C^{1}_{0}(\overline{\Omega})$ and $u>0$ in $\Omega$. Hence $\lambda=\hat{\lambda}_{1}$ (due to Theorem 5.1 in \cite{L}), but this clearly contradicts the assumption $\lambda>\hat{\lambda}_{1}$.  
	
	Finally we prove that $u\neq 0$. The following two cases may occur:\\
	CASE 1: $\mu\neq \hat{\mu}_{1}$.\\
	CASE 2: $\mu=\hat{\mu}_{1}$.
	
	The former follows as done above to reach $v\neq 0$. For the latter, assume by contradiction that $u=0$. Because $(0,v)$ is a weak solution of $(P_{\lambda,\hat{\mu}_{1}})$ and $v\neq 0$, we deduce $v=\epsilon \psi_{1}$ for some $\epsilon>0$. Thus $$0=I_{\sigma}(0,\epsilon \psi_{1})=I_{\sigma}(0,v)<0,$$
	which is a contradiction.
	
	Therefore, $u\neq 0$ and $v\neq 0$. From Proposition \ref{pro1} and Proposition \ref{pro2} one has $u,v\in int ~C^{1}_{0}(\overline{\Omega})_{+}$, that is, $(u,v)$ is a positive solution of \eqref{pq}. This finishes the proof.
	
\end{proof}

\begin{remark} Let us point out that to get $u>0, v>0$ in $\Omega$ in Theorem \ref{sub-sup} it was essential that $(\overline{\lambda},\overline{\mu})\notin \Lambda_{1}\cup \Lambda_{2}\cup X_{\hat{\lambda}_{1},\hat{\mu}_{1}}$ and $\sigma=(\lambda,\mu)\in (-\infty,\overline{\lambda}]\times (-\infty,\overline{\mu}]
	\setminus \left(\Lambda_{1}\cup \Lambda_{2}\cup X_{\hat{\lambda}_{1},\hat{\mu}_{1}}\right)$. Indeed, for $(\lambda,\mu)$ outside of this set we are not able to distinguish among the relative minimizer $(u,v)\in M$ of $I_{\sigma}$ and the weak solutions $(0,0), (\epsilon \phi_{1},0),\epsilon>0$ and $(0,\epsilon \psi_{1}),\epsilon>0$, in the cases where $(\lambda,\mu)$ belongs to $X_{\hat{\lambda}_{1},\hat{\mu}_{1}}, \Lambda_{2}$ and $\Lambda_{1}$, respectively.
	
\end{remark}

Since every positive solution of \eqref{pq} is also a supersolution, the following result holds true.
\begin{corollary}\label{corsub-sup}
	Let $(\overline{\lambda},\overline{\mu})\notin \Lambda_{1}\cup \Lambda_{2}\cup X_{\hat{\lambda}_{1},\hat{\mu}_{1}}$. Assume \eqref{pot} holds and suppose
	that  $(P_{\overline{\lambda},\overline{\mu}})$ has a positive solution. Then problem \eqref{pq} has a positive solution $(u,v)$ for each $\sigma=(\lambda,\mu)\in (-\infty,\overline{\lambda}]\times (-\infty,\overline{\mu}]
	\setminus \left(\Lambda_{1}\cup \Lambda_{2}\cup X_{\hat{\lambda}_{1},\hat{\mu}_{1}}\right)$. Moreover, one has $I_{\sigma}(u,v)<0$.
\end{corollary}


\section{Proof of Theorem \ref{T1}}
Let us denote by
\begin{equation}\label{up}
\Upsilon=\left\{(\lambda,\mu)\in \mathbb{R}^{2}:\eqref{pq}~\mbox{admits  at least one positive solution}\right\},
\end{equation}
\begin{equation*} ~~~~~~
\begin{array}{ll}
\overline{\Upsilon}:=~\mbox{the closure of}~\Upsilon,\\
\nonumber int(\Upsilon):=~\mbox{the interior of}~\Upsilon,\\
\partial (int(\Upsilon)):=~\mbox{the boundary of}~int(\Upsilon),\\
d(int(\Upsilon)):=~\mbox{the derived set of}~int(\Upsilon),\\
\overline{int(\Upsilon)}:=~\mbox{the closure of}~int(\Upsilon).
\end{array}
\end{equation*}

We will use the symbols $\hat{\lambda}_{1}(\tilde{\Omega})$ and $\hat{\mu}_{1}(\tilde{\Omega})$ for the first eigenvalues of the operators $-\Delta_{p}$ and $-\Delta_{q}$ in $\tilde{\Omega}$ with zero boundary conditions, respectively.  From now on we assume \eqref{pot} and $(f)_{1}-(f)_{2}$ hold.

The main purpose of this section is to study the properties of the set $\Upsilon$ and prove Theorem \ref{T1}. From \cite{BM,BI1,BI2,SM1} we deduce the following properties:
\begin{itemize}
	\item[$(\Upsilon)_{1}$] $X_{\hat{\lambda}_{1},\hat{\mu}_{1}} \subset int(\Upsilon)$, 
	\item[$(\Upsilon)_{2}$] there exists $\epsilon>0$ such that $]\hat{\lambda}_{1},\hat{\lambda}_{1}+\epsilon[\times]\hat{\mu}_{1},\hat{\mu}_{1}+\epsilon[\subset int(\Upsilon)$, 
	\item[$(\Upsilon)_{3}$] as a consequence of $(\Upsilon)_{1}-(\Upsilon)_{2}$ one has $\Upsilon\neq \emptyset, int(\Upsilon)\neq \emptyset$ and $int(\Upsilon) \cap \left\{(\lambda,\mu)\in \mathbb{R}^{2}:\hat{\lambda}_{1}<\lambda ~\mbox{and}~\hat{\mu}_{1}<\mu\right\}\neq \emptyset$.
\end{itemize}

Now, we find an upper bound for $\Upsilon$. 

\begin{lemma}\label{l0}
	If \eqref{pq} has a positive solution, then $\lambda \leq \hat{\lambda}_{1}(\tilde{\Omega})$ and $\mu \leq \hat{\mu}_{1}(\tilde{\Omega})$.
\end{lemma}

\begin{proof} Suppose that \eqref{pq} admits a positive solution $(u,v)$ and let $\varphi\in C^{\infty}_{0}(\tilde{\Omega})$ with $\varphi \geq 0$. From Proposition \ref{pro2} we have $u,v\in int ~C^{1}_{0}(\overline{\Omega})_{+} $. Hence, by Picone's identity (see \cite{AH}),
	$$\displaystyle \int_{\Omega} \vert \nabla \varphi \vert^{p}-\displaystyle \int_{\Omega}  \vert \nabla u \vert^{p-2}\nabla u\nabla \left(\varphi^{p}/u^{p-1}\right)\geq 0  $$
	and consequently, from the equation satisfied by $u$ and $(f)_{2}$,
	\begin{align*}
	\displaystyle \int_{\tilde{\Omega}} \vert \nabla \varphi \vert^{p}=\displaystyle \int_{\Omega} \vert \nabla \varphi \vert^{p}&\geq \displaystyle \int_{\Omega}\left(\lambda \vert u\vert^{p-2}u+\alpha f(x) \vert u\vert^{\alpha-2}\vert v\vert  ^{\beta}u\right)\varphi^{p}/u^{p-1}\\
	&= \displaystyle \int_{\tilde{\Omega}}\left(\lambda  u^{p-1}+\alpha f(x) \vert u\vert^{\alpha-2}\vert v\vert  ^{\beta}u\right)\varphi^{p}/u^{p-1}\\
	&\geq \lambda \displaystyle \int_{\tilde{\Omega}}  u^{p-1}\varphi^{p}/u^{p-1},
	\end{align*}
	that is
	$$\displaystyle \int_{\tilde{\Omega}} \vert \nabla \varphi \vert^{p}\geq \lambda \displaystyle \int_{\tilde{\Omega}} \varphi^{p}.$$
	Taking the infimum with respect to $\varphi$ yields $\hat{\lambda}_{1}(\tilde{\Omega})\geq \lambda$. Similarly, one can show that $\mu \leq \hat{\lambda}_{1}\left(\tilde{\Omega}\right)$. The proof of the lemma is complete.
\end{proof}

As a consequence of Lemma \ref{l0} we get
$$int(\Upsilon)\subset \Upsilon \subset (-\infty,\hat{\lambda}_{1}(\tilde{\Omega})]\times (-\infty,\hat{\mu}_{1}(\tilde{\Omega})].$$

Next, we define a family of straight lines
$$
L(t)=\left\{(\lambda,\lambda-t):\lambda \in \mathbb{R}\right\},~t\in \mathbb{R}
$$
and
$$
\lambda(t)=\sup \left\{\lambda :(\lambda,\lambda-t)\in  \overline{int(\Upsilon)}\right\},~\mu(t)=\lambda(t)-t~\mbox{and}~\Gamma(t)=(\lambda(t),\mu(t)).
$$

By Lemma \ref{l0}  one has the estimates 
\begin{align*} 
\left\{ \begin{array}{c}
\lambda(t)\leq \hat{\lambda}_{1}(\tilde{\Omega}),~ \mbox{for every}~ t\in \mathbb{R},\\
\mu(t)\leq \hat{\mu}_{1}(\tilde{\Omega}),~\mbox{for every}~t\in \mathbb{R},\\
\mu(t)\leq \hat{\lambda}_{1}(\tilde{\Omega})-t,~\mbox{for every}~t\in \mathbb{R}.
\end{array}
\right.&
\end{align*}

Thus, the functions $\lambda(t)$ and $\mu(t)$ are well defined, and consequently  $\Gamma(t)$ is well defined.

In order to apply Corollary \ref{corsub-sup} and Theorem \ref{sub-sup}, we will obtain estimates from below for $\Gamma$ in the next lemma.
\begin{lemma}\label{ll1} There exists
	$\theta\in \mathbb{R}$ such that $\lambda(t)>\hat{\lambda}_{1}$ for $t>\theta$ and $\mu(t)>\hat{\mu}_{1}$ for $t\leq\theta$.
\end{lemma}
\begin{proof}
	By $(\Upsilon)_{3}$	we can find $(\lambda,\mu)\in \Upsilon$  with $\lambda>\hat{\lambda}_{1}$ and $\mu>\hat{\mu}_{1}$.  Let us set $\theta=\lambda-\mu$. Then, it is easy to see that
	\begin{align*} 
	L(t)\cap \partial \left((-\infty,\lambda]\times (-\infty,\mu]\right)=	\left\{ \begin{array}{c}
	(\lambda,\lambda-t)~ \mbox{if}~ t> \theta,\\
	(\mu+t,\mu)~\mbox{if}~t\leq\theta,
	\end{array}
	\right.
	\end{align*}
	and hence, by Corollary \ref{corsub-sup},
	$$\hat{\lambda}_{1}<\lambda\leq\lambda(t)~\mbox{if}~t> \theta$$
	and 
	$$\hat{\mu}_{1}<\mu\leq \lambda(t)-t=\mu(t)~\mbox{if}~t\leq \theta.$$
	
	This concludes the proof.
	
\end{proof}

Furthermore, we have the following lemma.

\begin{lemma}\label{ll3}
	$\Gamma(t)\in \partial(int(\Upsilon))$ for every $t\in \mathbb{R}$.
\end{lemma}
\begin{proof} For any  $t\in \mathbb{R}$ given, by the definition of $\lambda(t)$ there exists a sequence $\left\{(\lambda_{k},\mu_{k})\right\}\subset L(t)\cap \overline{int(\Upsilon)}$ which converge to $(\lambda(t),\mu)$, for some $\mu \in \mathbb{R}$. Now, by the definition of $L(t)$ and this convergence, we have  $\mu_{k}=\lambda_{k}-t$ and $\mu=\displaystyle \lim_{k\to \infty}\mu_{k}=\lambda(t)-t=\mu(t)$. Hence, $(\lambda(t),\mu(t))=(\lambda(t),\mu)\in  \overline{int(\Upsilon)}$, that is, $\Gamma(t)\in \overline{int(\Upsilon)}$. We claim that $\Gamma(t)\notin int(\Upsilon)$. Indeed, if $\Gamma(t)\in int(\Upsilon)$, then  there would be a $r>0$ such that $B_{r}(\Gamma(t))\subset int(\Upsilon)$. Since $B_{r}(\Gamma(t))\cap L(t)\neq \emptyset $ and $f(\lambda)=\lambda-t $ is an increasing function, there exists $\lambda>\lambda(t)$ such that $(\lambda,\lambda-t)\in B_{r}(\Gamma(t))$ and so $(\lambda,\lambda-t)\in  \overline{int(\Upsilon)}$, which is a contradiction with the definition of $\lambda(t)$. Therefore $\Gamma(t)\in \partial (int(\Upsilon))$. This concludes the proof of the lemma.
	
\end{proof}

Let us set
$$\lambda_{\ast}=\displaystyle\sup_{t\in \mathbb{R}}\lambda(t)~\mbox{and}~\mu_{\ast}=\displaystyle\sup_{t\in \mathbb{R}}\mu(t).$$
In the next lemmas, we prove the main properties of $\Gamma$ and $\Upsilon$.

\begin{lemma}\label{L8}
	The following conclusions hold true:
	\begin{itemize}
		\item[$a)$] $\lambda (t)$ is monotone nondecreasing and $\mu(t)$ is monotone nonincreasing,
		\item[$b)$]  
		$\Gamma:\mathbb{R}\longrightarrow\mathbb{R}^{2}$ is a continuous function,
		\item[$c)$] $\displaystyle \lim_{t\to +\infty}\Gamma(t)=(\lambda_{\ast},-\infty)$ and $\displaystyle \lim_{t\to -\infty}\Gamma(t)=(-\infty,\mu_{\ast})$,
		\item[$d)$] $\Gamma (t)$ is injective,
		\item[$e)$] 
		\begin{equation}\label{119}
		\partial (int(\Upsilon))\setminus \left(\Lambda_{1}\cup \Lambda_{2}\right)=\left\{\Gamma (t): t\in \mathbb{R}\right\},
		\end{equation}
		
		\item[$f)$] the
		\begin{align}\label{1191}
		\overline{int(\Upsilon)}=\displaystyle \bigcup_{t\in \mathbb{R}}\left\{(\lambda,\mu)\in L(t): (\lambda,\mu) \leq \Gamma(t) \right\}.
		\end{align}
		
		
	\end{itemize}
	
\end{lemma}
\begin{proof} Firstly let us prove $a)$. Let us suppose by contradiction that $\lambda(t)$ is not monotone nondecreasing. Then there exist $t,s\in \mathbb{R}$, with $t<s$ and $\lambda(t)>\lambda(s)$. Thus, $\lambda(t)>\lambda(s)$ and $\mu(t)>\mu(s)$. Let $\lambda$ such that $\lambda(s)<\lambda<\lambda(t)$. Since
	$$\mu(s)=\lambda(s)-s<\lambda-s<\lambda-t<\lambda(t)-t=\mu(t),$$
	it follows from Lemma \ref{ll1}, Theorem \ref{sub-sup} and definition of $\Gamma(t)$ that  system $(P_{\lambda,\lambda-t})$ has a positive solution $(u,v)$, which is a supersolution of $(P_{\lambda,\lambda-s})$. So, Lemma \ref{ll1} and Theorem \ref{sub-sup}  imply that system $(P_{\lambda,\lambda-s})$ admits a solution $(\tilde{u},\tilde{v})$, which lead us to conclude that   $\lambda(s)<\lambda\leqslant \lambda(s)$, but this is a contradiction. Similarly, we deduce the monotonicity of $\mu(t)$.

	Now, let us to prove $b)$. Let $s,t\in \mathbb{R}$. Without loss of generality, we can suppose that $s<t$. Therefore, combining the definition of $\Gamma (t)$ with the monotone nondecrease of $\lambda(t)$ and the monotone nonincrease of $\mu(t)$, one has
	$$\vert \Gamma(s)-\Gamma(t)\vert \leq \vert \lambda(s)-\lambda(t)\vert+\vert \mu(s)-\mu(t)\vert=-\lambda(s)+\lambda(t)+\lambda(s)-s-\lambda(t)+t,$$
	namely,
	\begin{equation}\label{lip}
	\vert \Gamma(s)-\Gamma(t)\vert \leq \vert t-s\vert, \forall s,t \in \mathbb{R}.
	\end{equation}
	
	This means that $\Gamma$ is continuous. 	
	
	Let us prove the first statement of item $c)$. From $\lambda(t)\leq \hat{\lambda}_{1}(\tilde{\Omega})$ for every $t\in \mathbb{R}$ we obtain $\mu(t)=\lambda(t)-t\leq \hat{\lambda}_{1}(\tilde{\Omega})-t$ for every $t\in \mathbb{R}$, and passing to the limit as $t\to +\infty$, we have $\displaystyle \lim_{t\to +\infty}\mu(t)=-\infty$. This, item $a)$ and definition of $\lambda_{\ast}$ imply that 
	$\displaystyle \lim_{t\to +\infty}\Gamma(t)=(\lambda_{\ast},-\infty)$. 
	
	To prove the second statement, we note that $\mu(t)\leq \hat{\mu}_{1}(\tilde{\Omega})
	$ for every $t\in \mathbb{R}$ implies $\lambda(t)\leq \hat{\mu}_{1}(\tilde{\Omega})+t$ for every $t\in \mathbb{R}$, and passing to the limit as $t\to -\infty$, we have $\displaystyle \lim_{t\to -\infty}\lambda(t)=-\infty$. This, item $a)$ and definition of $\mu_{\ast}$ imply that 
	$\displaystyle \lim_{t\to -\infty}\Gamma(t)=(-\infty,\mu_{\ast})$. This completes the proof of the item $c)$.
	
	Now, let us prove $d)$. If $\Gamma(t)=\Gamma(s)$, then $\lambda(t)=\lambda(s)$ and $\lambda(t)-t=\lambda(s)-s$  that implies $t=s$. Therefore, $\Gamma$ is injective and this completes the proof of $d)$.
	
	Proof of $e)$. It follows from Lemmas \ref{ll3} and \ref{ll1}  that
	$$\left\{\Gamma (t): t\in \mathbb{R}\right\}\subset \partial (int(\Upsilon))\setminus \left(\Lambda_{1}\cup \Lambda_{2}\right)$$
	and so, to complete the proof, it suffices  to show
	$$\partial (int(\Upsilon))\setminus \left(\Lambda_{1}\cup \Lambda_{2}\right) \subset \left\{\Gamma (t): t\in \mathbb{R}\right\}.$$
	
	To do this, by letting
	\begin{equation}\label{E14}
	(a,b)\in \partial (int(\Upsilon))\setminus \left(\Lambda_{1}\cup \Lambda_{2}\right),
	\end{equation}
	we have that
	$$
	(a,b)\in L(t_{0})
	$$
	for $t_{0}=a-b$, whence together with  (\ref{E14}), we obtain  $(a,b)\in L(t_{0})\cap \overline{int(\Upsilon)}$. Moreover, by definition of $\lambda(t_{0})$, we have that $a\leq \lambda(t_{0})$. Therefore,
	$$\left\{(a,b),(\lambda (t_{0}),\mu (t_{0}))\right\}\subset L(t_{0})\ \text{and}\ a\leq \lambda (t_{0}).$$
	
	We are going to proof that $a=\lambda(t_{0})$. If $a< \lambda (t_{0})$, then $b< \mu (t_{0})$. By definition of $\Gamma (t_{0})$, there exists $\left\{(\lambda_{k},\mu_{k})\right\}\subset int(\Upsilon)$ such that $\lambda_{k}\rightarrow \lambda(t_{0})$ and $\mu_{k}\rightarrow \mu(t_{0})$ with $k\rightarrow +\infty$. Hence, there exists $k_{0}\in \mathbb{N}$ such that
	$$a< \lambda_{k_{0}}<\lambda(t_{0})\ \text{and}\ b< \mu_{k_{0}}<\mu(t_{0}),$$
	which implies, together with  Lemma \ref{ll1} and Theorem \ref{sub-sup}, that
	$$(a,b)\in(-\infty,\lambda_{k_{0}}[\times (-\infty,\mu_{k_{0}}[\setminus (-\infty,\hat{\lambda}_{1}]\times (-\infty,\hat{\mu}_{1}]\subset int(\Upsilon),$$
	that is, $(a,b)\in int(\Upsilon)$, but this is a contradiction with (\ref{E14}). So
	$$(a,b)=(\lambda (t_{0}),\mu (t_{0}))\in \left\{\Gamma (t): t\in \mathbb{R}\right\}$$
	that shows \eqref{119}. This finishes the proof of $e)$.
	
	
	Proof of $f)$. By definition, for any 
	$$(a,b)\in \displaystyle \bigcup_{t\in \mathbb{R}}\left\{(\lambda,\mu)\in L(t): (\lambda,\mu) \leq \Gamma(t)\right\}$$
	given, there exists a $t\in \mathbb{R}$ such that
	\begin{equation}\label{120}
	(a,b)\in L(t),~a \leq \lambda(t)~\mbox{and}~ b\leq\mu(t).
	\end{equation}
	
	In view of Lemma \ref{ll3}, $(\lambda(t),\mu(t))\in L(t)\cap \partial(int(\Upsilon))$. Let  $(\lambda,\mu)<(\lambda(t),\mu(t))$. So, by Lemma \ref{ll1} there exists $(\kappa,\xi)\in int(\Upsilon)$ such that $(\lambda,\mu)<(\kappa,\xi)$ and either $\hat{\lambda}_{1}<\kappa$ or $\hat{\mu}_{1}<\xi$, which implies by Theorem \ref{sub-sup}  that $(\lambda,\mu)\in \left(int(\Upsilon)\cup \Lambda_{1}\cup \Lambda_{2}\right)\subset \overline{int(\Upsilon)}$. Therefore $(-\infty,\lambda(t)]\times (-\infty,\mu(t)]\subset \overline{int(\Upsilon)}$ and by \eqref{120} we have $(a,b)\in \overline{int(\Upsilon)}$. This means that
	\begin{equation}\label{144}
	\bigcup_{t\in \mathbb{R}}\left\{(\lambda,\mu)\in L(t): (\lambda,\mu) \leq \Gamma(t)\right\}\subset \overline{int(\Upsilon)}.
	\end{equation}

	To end the proof, we claim that			
	\begin{align}\label{148}
	\overline{int(\Upsilon)}\subset \bigcup_{t\in \mathbb{R}}\left\{(\lambda,\mu)\in L(t):(\lambda,\mu) \leq \Gamma(t)\right\}.
	\end{align}
	
	Indeed, for any $(a,b)\in \overline{int(\Upsilon)}$, we obtain   $(a,b)=(a,a-(a-b))\in L(t)$, where $t=a-b$ and so $(a,b)\in L(t)\cap \overline{int(\Upsilon)}$. By the definitions of $\lambda(t)$ and $\mu(t)$, we have $a\leq \lambda(t)$ and $b\leq \mu(t)$. Hence, $(a,b)\in L(t), a\leq \lambda (t)$ and $b \leq \mu(t)$, that is,
	$$
	(a,b)\in \bigcup_{t\in \mathbb{R}}\left\{(\lambda,\mu)\in L(t): (\lambda,\mu) \leq \Gamma(t)\right\}.
	$$
	Thus, the claim follows. Combining  \eqref{144} and \eqref{148} we get \eqref{1191}. The proof of lemma is now complete.
	
\end{proof}

\begin{lemma}\label{l9} The following properties hold true:
	\begin{itemize}
		\item[$a)$]$\partial(\Upsilon)=\partial(int(\Upsilon))$. 
		\item[$b)$] $\overline{\Upsilon}=\overline{int(\Upsilon)}$.
	\end{itemize}
\end{lemma}
\begin{proof}
	$a)$	It is well
	known that $\partial(int(\Upsilon))\subset \partial(\Upsilon)$. Then it is sufficient to show that
	$\partial(\Upsilon)\subset \partial(int(\Upsilon))$. Indeed, let $(\lambda,\mu)\in \partial (\Upsilon)$ be
	arbitrary and $r>0$. If $(\kappa,\xi)\in \Upsilon \cap B_{r}((\lambda,\mu))$, then by Corollary \ref{corsub-sup} and $(\Upsilon)_{1}$ one has $(\kappa,\xi)\in \overline{int(\Upsilon)}$ and as a consequence, because $B_{r}((\lambda,\mu))$ is an open set, there exists $(\tilde{\kappa},\tilde{\xi})\in int(\Upsilon)\cap B_{r}((\lambda,\mu))$. Thus, $int(\Upsilon)\cap B_{r}((\lambda,\mu))\neq \emptyset$. On the other hand, since $\Upsilon^{c}\subset \left(int(\Upsilon)\right)^{c}$ and $B_{r}(\lambda,\mu)\cap \Upsilon^{c}\neq \emptyset$ we get $B_{r}(\lambda,\mu)\cap \left(int(\Upsilon)\right)^{c}\neq \emptyset$. Because $r>0$ was arbitrary, we conclude that $(\lambda,\mu)\in \partial(int(\Upsilon)) $ and $\partial(\Upsilon)\subset \partial(int(\Upsilon))$. Therefore $\partial(\Upsilon)=\partial(int(\Upsilon))$. \newline
	$b)$ Obviously $\overline{int(\Upsilon)}\subset \overline{\Upsilon}$. Moreover, we  proceed as in
	the proof of $a)$ to deduce that $int(\Upsilon)\cap B_{r}((\lambda,\mu))\neq \emptyset$ for every $(\lambda,\mu)\in \overline{\Upsilon}$ and $r>0$, and as a result we have  $\overline{\Upsilon}\subset \overline{int(\Upsilon)}$. Therefore $ \overline{\Upsilon}= \overline{int(\Upsilon)}$. The proof of lemma is now complete.
	
\end{proof}

We are now ready to prove Theorem \ref{T1}.

\textbf{\it Proof of Theorem \ref{T1}.} Let us define 
$$\Theta_{1}=int(\Upsilon)~\mbox{and}~\Theta_{2}=\mathbb{R}^{2}\backslash \overline{\Upsilon}.$$

By definition, system \eqref{pq} has  at least one positive solution for every $(\lambda,\mu)\in \Theta_{1}$, and in particular $ \Theta_{1}\backslash \left(\Lambda_{1}\cup \Lambda_{2}\right)$ shares the same property. On the other hand, since $\Upsilon\subset \overline{\Upsilon}$, we infer that system \eqref{pq} has no positive solution for every $(\lambda,\mu)\in \Theta_{2}$. Moreover,  from Lemma \ref{L8} $e)$ one has $\partial \Theta_{1}\backslash \left(\Lambda_{1}\cup \Lambda_{2}\right)=\Gamma$. This concludes the proof of the theorem.


\begin{remark}
	We do not know whether $\Lambda_{1}\cup \Lambda_{2}\subset \Upsilon$. If this is true,  from Lemma \ref{L8}-$e)$ and Lemma \ref{l9} $a)-b)$  one has $\partial(\Upsilon)=\partial(int(\Upsilon))=\Gamma$. In Theorem \ref{T2} we will show that if $p=q=2$, then Problem $(P_{\hat{\lambda}_{1},\hat{\lambda}_{1}})$ admits at least one positive solution, and hence $(\Lambda_{1}\cup \Lambda_{2})\cap \Upsilon\neq \emptyset$. 
\end{remark}
\begin{remark}
	Also, we do not know whether $\Gamma\subset \Upsilon$. If this is true, note that from  Lemma \ref{l9} $a)-b)$  one has $\Upsilon=\overline{int(\Upsilon)}$. In particular, $\Upsilon$ is a closed set.
\end{remark}

\section{Proof of Theorem \ref{T2}}

In this section we will prove Theorem \ref{T2}. Let us assume $p=q=2$, $2<\alpha+\beta<2^{\ast}$ and $f\in C(\overline{\Omega})$. Consider the problem
\begin{equation*}\label{p}\tag{$P_{\lambda}$}
\left\{
\begin{aligned}
-\Delta u = \lambda u+f(x) \vert u\vert^{\alpha+\beta-2}u~in ~ \Omega ,\\
u>0 ~\mbox{in} ~\Omega\\
u=0~\mbox{on}~\Omega, \nonumber
\end{aligned}
\right.
\end{equation*}

Let $(f)_{1}-(f)_{2}$ hold. From Alama-Tarantello \cite{AL} there exists $\tau>\hat{\lambda}_{1}$ such that:
\begin{itemize}
	\item[$a)$] for every $\lambda\in (\hat{\lambda}_{1},\tau)$, \eqref{p} admits at least two positive solutions.
	\item[$b)$] For $\lambda=\hat{\lambda}_{1}$ and $\lambda=\tau$, problem \eqref{p} admits at least one positive solution.
	\item[$c)$] For $\lambda>\tau$  problem \eqref{p} does not admit any positive solutions.
\end{itemize}
By using this result we have.

\textbf{\it Proof of Theorem \ref{T2} }
$a)$ Let $u_{\lambda}$ and $v_{\lambda}$  be two distinct positive solutions of \eqref{p}. It is easy to see that $(u_{\lambda},u_{\lambda})$ is a positive solution of the problem
\begin{equation*}\label{q}\tag{$Q_{\lambda}$}
\left\{
\begin{aligned}
-\Delta u_{\lambda} = \lambda  u_{\lambda}+ f(x) \vert u_{\lambda}\vert^{\alpha-2}\vert u_{\lambda}\vert ^{\beta}u_{\lambda}~in ~ \Omega ,\\
-\Delta u_{\lambda} =  \lambda u_{\lambda}+f(x)\vert u_{\lambda}\vert ^{\alpha}\vert u_{\lambda}\vert^{\beta-2}u_{\lambda}~in~ \Omega ,\\
u_{\lambda}>0 ~\mbox{in} ~\Omega\\
u_{\lambda}=0~\mbox{on}~\Omega, \nonumber
\end{aligned}
\right.
\end{equation*}

Multiplying the first equation of \eqref{q} by $t$ and the second equation by $s$, where $t,s>0$, we get (in the weak sense)
\begin{equation*}
\left\{
\begin{aligned}
-\Delta (tu_{\lambda}) = \lambda  (tu_{\lambda})+ t^{2-\alpha}s^{-\beta}f(x) \vert tu_{\lambda}\vert^{\alpha-2}\vert su_{\lambda}\vert ^{\beta}tu_{\lambda}~in ~ \Omega ,\\
-\Delta (su_{\lambda}) =  \lambda (su_{\lambda})+t^{-\alpha}s^{2-\beta}f(x)\vert tu_{\lambda}\vert ^{\alpha}\vert su_{\lambda}\vert^{\beta-2}su_{\lambda}~in~ \Omega ,\\
tu_{\lambda},su_{\lambda}>0 ~\mbox{in} ~\Omega,\\
tu_{\lambda}=su_{\lambda}=0~\mbox{on}~\Omega. \nonumber
\end{aligned}
\right.
\end{equation*}	

Choosing $t,s>0$ such that
$$t^{2-\alpha}s^{-\beta}=\alpha~\mbox{and}~t^{-\alpha}s^{2-\beta}=\beta$$
we find
$$t=\left(\alpha\right)^{\frac{\beta-2}{4d}}\left(\beta\right)^{\frac{-\beta}{4d}},~~~~~s=\left(\alpha\right)^{\frac{-\alpha}{4d}}\left(\beta\right)^{\frac{\alpha-2}{4d}},$$
where $d=\frac{\alpha}{2}+\frac{\beta}{2}-1\neq 0$, by the assumption. Hence $(\tilde{u}_{\lambda},\tilde{u}_{\lambda})=(tu_{\lambda},su_{\lambda})$ is a positive solution of $(P_{\lambda,\lambda})$. Similarly, we can prove that $(\tilde{v}_{\lambda},\tilde{v}_{\lambda})=(tv_{\lambda},sv_{\lambda})$ is a positive solution of $(P_{\lambda,\lambda})$ with $(\tilde{v}_{\lambda},\tilde{v}_{\lambda})\neq (\tilde{u}_{\lambda},\tilde{u}_{\lambda})$.  This completes the proof of  $a)$.

Arguing as in the proof of $a)$ we can prove $b)$. This finishes the proof of Theorem \ref{T2}.

{\it\small Ricardo Lima Alves}\\{\it\small  Departamento de Matem\'atica}\\ {\it\small Universidade de Bras\'ilia }\\
{\it\small 70910-900 Bras\'ilia}\\{\it\small DF - Brasil}\\{\it\small e-mail: ricardoalveslima8@gmail.com}\vspace{1mm}

\end{document}